\documentclass{amsart}

\usepackage{amsfonts, amsmath, amscd}
\usepackage[psamsfonts]{amssymb}

\usepackage{amssymb}

\usepackage{pb-diagram}

\usepackage[all,cmtip]{xy}

\usepackage[usenames]{color}

\newtheorem{theorem}{Theorem}[section]
\newtheorem{corollary}[theorem]{Corollary}
\newtheorem{definition}[theorem]{Definition}
\newtheorem{proposition}[theorem]{Proposition}
\newtheorem{question}[theorem]{Question}

\newtheorem{lemma}[theorem]{Lemma}
\newtheorem{fact}[theorem]{Fact}
\newtheorem{problem}[theorem]{Problem}

\newtheorem{example}[theorem]{Example}

\numberwithin{equation}{section}
\numberwithin{theorem}{section}

\newcommand{\TT}{{\mathbb T}}

\newcommand{\ZZ}{{\mathbb Z}}

\newcommand{\NN}{{\mathbb N}}

\newcommand{\uuu}{{\mathbf u}}
\newcommand{\vvv}{{\mathbf v}}

\newcommand{\Tub}{{\tau}_{b\mathbf{u}}}

\newcommand{\Tu}{{\tau}_\mathbf{u}}
\newcommand{\Tv}{{\tau}_\mathbf{v}}

\title[On $T$-characterized subgroups of compact Abelian groups]{On $T$-characterized subgroups of compact Abelian groups}

\author{ S. Gabriyelyan}
\address{S. Gabriyelyan: Department of Mathematics, Ben-Gurion University of the Negev, Beer-Sheva, P.O. 653, Israel}
\email{saak@math.bgu.ac.il}

\begin{document}

\begin{abstract}
We say that a subgroup $H$ of an infinite compact Abelian group $X$  is {\it $T$-characterized} if there is a $T$-sequence $\mathbf{u} =\{ u_n \}$ in the dual group of $X$ such that $H=\{ x\in X: \; (u_n, x)\to 1 \}$.   We show that a closed subgroup $H$ of $X$ is $T$-characterized if and only if $H$ is a $G_\delta$-subgroup  of $X$ and the annihilator of $H$ admits a Hausdorff minimally almost periodic group topology. All  closed subgroups of an infinite compact Abelian group $X$ are  $T$-characterized if and only if $X$ is metrizable and connected.
We prove that every compact Abelian group $X$ of infinite exponent has a $T$-characterized subgroup which is not an $F_{\sigma}$-subgroup of $X$ that gives a negative answer to Problem 3.3 in \cite{DG}.

\end{abstract}

\maketitle

\section{Introduction}

{\bf Notation and Preliminaries.} Let $X$ be an Abelian topological group.  We   denoted by $\widehat{X}$ the group of all continuous characters on $X$, $\widehat{X}$ endowed with the compact-open topology is denoted by $X^{\wedge}$. The homomorphism $\alpha_X : X\to X^{\wedge\wedge} $, $x\mapsto (\chi\mapsto (\chi, x))$, is called {\it the canonical homomorphism}.  Denote by $\mathbf{n}(X) = \cap_{\chi\in \widehat{X}} {\rm ker}(\chi) = \ker(\alpha_X)$ the von Neumann radical of $X$. The group $X$ is called {\it minimally almost periodic} ($MinAP$) if $\mathbf{n}(X) = X$, and $X$ is called {\it maximally almost periodic} ($MAP$) if $\mathbf{n}(X) = \{ 0\}$. Let $H$ be a subgroup of $X$. The {\it annihilator} of $H$ we denote by $H^\perp$, i.e., $H^\perp =\{ \chi \in X^\wedge : (\chi, h)=1 \mbox{ for every } h\in H\}$.

Recall that an Abelian group  $G$ is {\it of finite exponent} or {\it bounded} if there exists a positive integer $n$ such that $ng=0$ for every $g\in G$. The minimal integer $n$ with this property is called the {\it exponent} of $G$ and is denoted by $\mathrm{exp}(G)$. When $G$ is not bounded, we write $\mathrm{exp}(G)=\infty$ and say that $G$ is {\it of infinte exponent}  or {\it unbounded}.
 The direct sum of $\omega$ copies of an Abelian group $G$ we denote by $G^{(\omega)}$.

Let $\mathbf{u}=\{ u_n\}_{n\in\omega}$ be a sequence in an Abelian group  $G$. In general no Hausdroff topology may exist in which  $\uuu$ converges to zero. A  very important question whether there exists a Hausdorff group topology $\tau$ on $G$ such that $u_n \to 0$ in $(G,\tau)$, especially for the integers, has been studied by many authors, see  Graev \cite{Gra}, Nienhuys \cite{Nie}, and others. Protasov and  Zelenyuk \cite{ZP1} obtained a criterion that gives a complete answer to this question. Following  \cite{ZP1}, we say that a sequence $\mathbf{u} =\{ u_n \}$ in an Abelian group $G$ is a $T$-{\it sequence} if there is a Hausdorff group topology on $G$ in which $u_n $ converges to zero. The finest group topology with this property we denote by $\Tu$.

The counterpart of the above question for {\it precompact} group topologies on $\mathbb{Z}$ is studied by Raczkowski \cite{Rac}. Following \cite{BDM, BDMW} and motivated by \cite{Rac}, we say that a sequence $\mathbf{u} =\{ u_n\}$ is a $TB$-{\it sequence} in an Abelian group $G$ if there is a precompact Hausdorff group topology on $G$ in which  $u_n $ converges to zero. For a $TB$-sequence $\mathbf{u}$ we denote by $\Tub$ the finest precompact group topology  on $G$ in which  $\mathbf{u}$ converges to zero. Clearly, every $TB$-sequence is a $T$-sequence, but in general, the converse assertion does not hold.

While it is quite hard to check whether a given sequence is a $T$-sequence (see, for example, \cite{Ga5,GaFin,Ga4,ZP1,ZP2}), the case of $TB$-sequences is much simpler. Let $X$ be an Abelian topological group and $\mathbf{u} =\{ u_n \}$ be a sequence in its dual group $X^{\wedge}$. Following \cite{DMT}, set
\[
s_{\mathbf{u}} (X)=\{ x\in X: \; (u_n, x)\to 1 \}.
 \]
In \cite{BDM} the following simple criterion to be a $TB$-sequence was obtained:
\begin{fact} {\rm  \cite{BDM}} \label{f1}
A sequence $\uuu$ in a (discrete) Abelian group $G$  is a $TB$-sequence if and only if the subgroup $s_{\mathbf{u}} (X)$ of the (compact) dual $X=G^\wedge$ is dense.
\end{fact}
Motivated by Fact \ref{f1}, Dikranjan et al. \cite{DMT} introduced the following notion related to subgroups of the form $s_{\mathbf{u}} (X)$ of a compact Abelian group $X$:
\begin{definition} {\rm  \cite{DMT}} \label{d1}
Let $H$ be a subgroup of a compact Abelian group $X$ and $\mathbf{u} =\{ u_n \}$ be a sequence in  $\widehat{X}$. If $H=s_{\mathbf{u}} (X)$ we say that $\mathbf{u}$ {\em characterizes} $H$ and that $H$ is {\em characterized} (by $\mathbf{u}$).
\end{definition}
Note that for the torus $\TT$ this notion was already defined in \cite{BDS}.
Characterized subgroups has been studied by many authors, see, for example, \cite{BSW, BDS, DG, DiK, DMT, Ga}. In particular, the main theorem of \cite{DiK} (see also \cite{BSW})  asserts that every countable subgroup of a compact metrizable Abelian group is characterized. It is natural to ask whether a closed subgroup of a compact Abelian group is characterized. The following easy criterion is given in \cite{DG}:
\begin{fact}  {\em \cite{DG}} \label{f2}
A closed subgroup $H$ of a compact Abelian group $X$ is characterized if and only if $H$ is a $G_\delta$-subgroup. In particular, $X/H$ is metrizable and the annihilator $H^\perp$ of $H$ is countable.
\end{fact}

The next fact  follows easily from Definition \ref{d1}:

\begin{fact} {\em (\cite{CRT}, see also \cite{DG})} \label{f21}
Every characterized subgroup $H$  of a compact Abelian group $X$ is an $F_{\sigma\delta}$-subgroup of $X$, and hence $H$ is a Borel subset of $X$.
\end{fact}

Facts \ref{f2} and \ref{f21} inspired in \cite{DG} the study of the Borel hierarchy of characterized subgroups of compact Abelian groups. For a compact Abelian group $X$ denote by $\mathrm{Char} (X)$ (respectively, $\mathrm{SF}_\sigma(X)$, $\mathrm{SF}_{\sigma\delta}(X)$ and $\mathrm{SG}_\delta(X)$) the set of all characterized subgroups (respectively, $F_\sigma$-subgroups, $F_{\sigma\delta}$-subgroups and $G_\delta$-subgroups) of $X$. The next fact is Theorem E in  \cite{DG}:
\begin{fact} {\em  \cite{DG}} \label{fBor}
For every infinite compact Abelian group $X$, the following inclusions hold:
\[
\mathrm{SG}_\delta(X) \subsetneqq \mathrm{Char}(X) \subsetneqq \mathrm{SF}_{\sigma\delta}(X)\ \  \mbox{ and } \  \  \mathrm{SF}_\sigma(X) \not \subseteq  \mathrm{Char}(X).\]
If in addition $X$ has finite exponent, then
\begin{equation} \label{Bor}
\mathrm{Char}(X)  \subsetneqq \mathrm{SF}_\sigma(X).
\end{equation}
\end{fact}
The inclusion (\ref{Bor}) inspired the following question:
\begin{question} {\rm \cite[Problem 3.3]{DG}} \label{q2}
Does there exist a compact Abelian group $X$ of {\rm infinite} exponent whose all characterized subgroups are $F_\sigma$-subsets of $X$?
\end{question}

{\bf Main results.} It is important to emphasize that there is no any restriction on a sequence $\uuu$ in Definition \ref{d1}. If a characterized subgroup $H$ of a compact Abelian group $X$ is dense, then, by Fact \ref{f1}, a characterizing sequence is also a $TB$-sequence. But if $H$ is not dense,  we can not expect in general that a characterizing sequence of $H$ is a $T$-sequence. Thus it is natural to ask:
\begin{question} \label{q1}
For which characterized subgroups of compact Abelian groups one can find  characterizing sequences which are also  $T$-sequences?
\end{question}

This question is of independent interest because every $T$-sequence $\uuu$ naturally defines the group topology $\Tu$ satisfying the following   dual property:
\begin{fact}  {\em \cite{Ga3}} \label{f4}
Let $H$ be a characterized subgroup of  an infinite compact  Abelian group $X$ by  a $T$-sequence $\uuu$. Then  $(\widehat{X},\Tu)^\wedge =H(=s_\mathbf{u} (X))$ and $\mathbf{n}(\widehat{X},\Tu)=H^\perp$ algebraically.
\end{fact}

 This motivates us to introduce the following notion:

\begin{definition} \label{d2}
Let $H$ be a subgroup of a compact Abelian group $X$. We say that $H$ is a {\em $T$-characterized  subgroup} of $X$ if there exists a $T$-sequence $\uuu =\{ u_n\}_{n\in\omega}$ in $\widehat{X}$ such that  $H= s_\uuu (X)$.
\end{definition}
Denote by $\mathrm{Char}_T (X)$  the set of all $T$-characterized subgroups  of a compact Abelian group $X$. Clearly, $\mathrm{Char}_T (X)\subseteq \mathrm{Char}(X)$. Hence, if a $T$-characterized subgroup $H$ of $X$ is  closed it is a $G_\delta$-subgroup of $X$ by Fact \ref{f2}. Note also that $X$ is $T$-characterized by the zero sequence.

The main goal of the article is to obtain a complete description of closed $T$-characterized subgroups (see Theorem \ref{t1}) and to study the Borel hierarchy of $T$-characterized subgroups (see Theorem \ref{tS})  of compact Abelian groups. In particular, we obtain a complete answer to Question \ref{q1} for closed characterized  subgroups and give  a negative answer to   Question \ref{q2}.

Note that, if a compact Abelian group $X$ is finite, then every $T$-sequence $\uuu$ in $\widehat{X}$ is eventually equal to zero. Hence $s_\uuu(X)=X$. Thus $X$ is the unique $T$-characterized subgroup of $X$. So in what follows we shall consider only infinite compact groups.

The following theorem describes all closed    subgroups of compact Abelian groups which are $T$-characterized.
\begin{theorem} \label{t1}
Let $H$ be a proper closed subgroup of an infinite compact Abelian group $X$. Then the following assertions are equivalent:
\begin{enumerate}
\item[{\rm (1)}] $H$ is a $T$-characterized subgroup of $X$;
\item[{\rm (2)}] $H$ is a $G_\delta$-subgroup of $X$ and the countable group $H^\perp$ admits a Hausdorff MinAP group topology;
\item[{\rm (3)}] $H$ is a $G_\delta$-subgroup of $X$ and one of the following holds:
\begin{enumerate}
\item $H^\perp$ has infinite exponent;
\item $H^\perp$ has finite exponent and contains a subgroup which is isomorphic to $\mathbb{Z}\left(\exp(H^\perp)\right)^{(\omega)}$.
\end{enumerate}
\end{enumerate}
\end{theorem}

\begin{corollary} \label{c1}
Let $X$ be an infinite compact metrizable Abelian group. Then the trivial subgroup  $H=\{ 0\}$ is  $T$-characterized if and only if $\widehat{X}$ admits a Hausdorff MinAP group topology.
\end{corollary}

As an immediate corollary of Fact \ref{f2} and Theorem \ref{t1} we obtain
a complete answer to Question \ref{q1} for closed characterized  subgroups.
\begin{corollary} \label{c11}
A proper closed characterized subgroup $H$ of an infinite  compact Abelian group $X$ is  $T$-characterized if and only if $H^\perp$ admits a Hausdorff MinAP group topology.
\end{corollary}

If  $H$ is an open proper subgroup of $X$, then $H^\perp$ is non-trivial and finite. Thus every Hausdorff group topology on $H^\perp$ is discrete. Taking into account Fact \ref{f2} we obtain:
\begin{corollary} \label{c2}
Every open proper subgroup $H$ of an infinite  compact Abelian group $X$ is a characterized   non-$T$-characterized subgroup of $X$.
\end{corollary}
Nevertheless  (see Example \ref{exa1} below) there is a   compact metrizable Abelian group $X$ with a countable $T$-characterized subgroup $H$ such that its closure $\bar{H}$ is open.   Thus it may happened that the closure of a $T$-characterized subgroup is not  $T$-characterized.

It is natural to ask for which compact Abelian groups {\it all} their closed $G_\delta$-subgroups are $T$-characterized. The next theorem gives a complete answer to this question.

\begin{theorem} \label{t2}
Let $X$ be an infinite  compact Abelian group. The following  assertions are equivalent:
\begin{enumerate}
\item[{\rm (1)}] All closed $G_\delta$-subgroups of $X$ are $T$-characterized;
\item[{\rm (2)}] $X$ is connected.
\end{enumerate}
\end{theorem}
By  Corollary 2.8 of \cite{DG}, the trivial subgroup  $H=\{ 0\}$ of  a compact Abelian group $X$ is a $G_\delta$-subgroup if and only if $X$ is metrizable. So we obtain:
\begin{corollary} \label{c3}
All  closed subgroups of an infinite compact Abelian group $X$ are  $T$-characterized if and only if $X$ is metrizable and connected.
\end{corollary}

Theorems \ref{t1} and \ref{t2} are proved in Section \ref{s1}.

In the next theorem we give a negative answer to Question \ref{q2}:
\begin{theorem} \label{t3}
Every compact Abelian group of infinite exponent has a dense $T$-characterized subgroup which is not an $F_\sigma$-subgroup.
\end{theorem}

As a corollary of the inclusion (\ref{Bor}) and Theorem \ref{t3} we obtain:
\begin{corollary} \label{c4}
For an infinite compact Abelian group $X$ the following  assertions are equivalent:
\begin{enumerate}
\item[{\rm (i)}] $X$ has finite exponent;
\item[{\rm (ii)}] every characterized subgroup of $X$ is an $F_\sigma$-subgroup;
\item[{\rm (iii)}] every $T$-characterized subgroup of $X$ is an $F_\sigma$-subgroup.
\end{enumerate}
Therefore, $\mathrm{Char}(X)  \subseteq \mathrm{SF}_\sigma(X)$ if and only if $X$ has finite exponent.
\end{corollary}

In the next theorem we summarize the obtained results about the Borel hierarchy of $T$-characterized subgroups of  compact Abelian groups.
\begin{theorem} \label{tS}
Let $X$ be an infinite compact Abelian group $X$. Then:
\begin{enumerate}
\item[{\rm (1)}] $\mathrm{Char}_T(X)  \subsetneqq \mathrm{SF}_{\sigma\delta}(X)$;
\item[{\rm (2)}] $\mathrm{SG}_{\delta}(X)\bigcap \mathrm{Char}_T(X)  \subsetneqq \mathrm{Char}_T(X)$;
\item[{\rm (3)}] $\mathrm{SG}_{\delta}(X) \subseteq \mathrm{Char}_T(X)$ if and only if $X$ is connected;
\item[{\rm (4)}] $\mathrm{Char}_T(X)\bigcap \mathrm{SF}_{\sigma}(X)  \subsetneqq  \mathrm{SF}_{\sigma}(X)$;
\item[{\rm (5)}] $\mathrm{Char}_T(X) \subseteq  \mathrm{SF}_{\sigma}(X)$ if and only if $X$ has finite exponent.
\end{enumerate}
\end{theorem}

We prove Theorems \ref{t3} and \ref{tS} in Section \ref{s2}.

The notions of $\mathfrak{g}$-closed and $\mathfrak{g}$-dense subgroups  of a compact Abelian group $X$ were  defined in \cite{DMT}. In the last section of the paper, in analogy to these notions, we define  $\mathfrak{g}_T$-closed and $\mathfrak{g}_T$-dense subgroups of $X$. In particular, we show that every $\mathfrak{g}_T$-dense subgroup of a compact Abelian group $X$ is dense if and only if $X$ is connected (see Theorem \ref{t41}).

\section{The Proofs of Theorems  \ref{t1} and \ref{t2}} \label{s1}

The subgroup of a group $G$ generated by a subset $A$ we denote by $\langle A\rangle$.

Recall that a subgroup $H$ of an Abelian topological group $X$ is called {\it dually closed} in $X$ if for every $x\in X\setminus H$ there exists a character $\chi \in H^\perp$ such that $(\chi,x)\not= 1$. $H$ is called {\it dually embedded} in $X$ if every character of $H$ can be extended to a character of $X$. Every open subgroup of $X$ is dually closed and dually embedded in $X$ by Lemma 3 of \cite{Nob}.

The next notion generalizes the notion of the maximal extension in the class of all compact Abelian groups introduced in \cite{DGT}.
\begin{definition} \label{dME}
Let $\mathcal{G}$ be an arbitrary class of topological groups. Let $(G,\tau)\in \mathcal{G}$ and $H$ be a subgroup of $G$. The group $(G,\tau)$ is called a {\em maximal extension  of $(H,\tau|_H)$ in the class $\mathcal{G}$} if $\sigma \leq \tau$ for every
group topology on $G$ such that $\sigma|_H =\tau|_H$ and $(G,\sigma)\in \mathcal{G}$.
\end{definition}
Clearly, the maximal extension is unique if it exists. Note that in Definition \ref{dME} we do not assume that $(H,\tau|_H)$ belongs to the class $\mathcal{G}$.

If $H$ is a subgroup of an Abelian group $G$ and $\uuu$ is a $T$-sequence (respectively, a $TB$-sequence) in $H$, we  denote by $\Tu(H)$ (respectively, $\Tub(H)$) the finest (respectively, precompact) group topology on $H$ generated by $\uuu$.
We use the following easy corollary of the definition of $T$-sequences.
\begin{lemma} \label{l11}
For a sequence $\uuu$ in an Abelian group $G$ the following  assertions are equivalent:
\begin{enumerate}
\item[{\rm (1)}] $\uuu$ is a $T$-sequence in $G$;
\item[{\rm (2)}] $\uuu$ is a $T$-sequence in every subgroup of $G$ containing $\langle\uuu\rangle$;
\item[{\rm (3)}] $\uuu$ is a $T$-sequence in  $\langle\uuu\rangle$.
\end{enumerate}
In this case, $\langle\uuu\rangle$ is open in $\Tu$ (and hence $\langle\uuu\rangle$ is dually closed and dually embedded in $(G, \Tu)$), and $(G, \Tu)$  is the maximal extension of  $(\langle\uuu\rangle, \Tu(\langle\uuu\rangle)$ in the class $\mathbf{TAG}$ of all Abelian topological groups.
\end{lemma}

\begin{proof}
Evidently, (1) implies (2) and  (2) implies (3). Let $\uuu$ be a $T$-sequence in  $\langle\uuu\rangle$. 
Let $\tau$ be the topology on $G$ whose base is all translates of $\Tu(\langle\uuu\rangle)$-open sets.  Clearly, $\uuu$ converges to zero in $\tau$. Thus $\uuu$ is a $T$-sequence in $G$. So (3) implies (1).

Let us prove the last assertion. By the definition of $\Tu$ we have also $\tau\leq\Tu$, and hence $\tau|_{\langle\uuu\rangle}= \Tu (\langle\uuu\rangle)\leq \Tu|_{\langle\uuu\rangle}$. Thus $\langle\uuu\rangle$ is open in $\Tu$, and hence it is dually closed and dually embedded in $(G, \Tu)$ by \cite[Lemma 3.3]{Nob}. On the other hand,  $\Tu|_{\langle\uuu\rangle} \leq \Tu (\langle\uuu\rangle)=\tau|_{\langle\uuu\rangle}$ by the definition of $\Tu (\langle\uuu\rangle)$. So $\Tu$ is an extension of $\Tu(\langle\uuu\rangle)$. Now clearly, $\tau =\Tu$ and $(G, \Tu)$  is the maximal extension of  $(\langle\uuu\rangle, \Tu(\langle\uuu\rangle)$ in the class $\mathbf{TAG}$.
\end{proof}

For  $TB$-sequences we have the following:
\begin{lemma} \label{l12}
For a sequence $\uuu$ in an Abelian group $G$ the following assertions are equivalent
\begin{enumerate}
\item[{\rm (1)}] $\uuu$ is a $TB$-sequence in $G$;
\item[{\rm (2)}] $\uuu$ is a $TB$-sequence in every subgroup of $G$ containing $\langle\uuu\rangle$;
\item[{\rm (3)}] $\uuu$ is a $TB$-sequence in  $\langle\uuu\rangle$.
\end{enumerate}
In this case, the subgroup $\langle\uuu\rangle$ is dually closed and dually embedded in $(G,\Tub)$, and $(G,\Tub)$  is the maximal extension  of $(\langle\uuu\rangle, \Tub(\langle\uuu\rangle))$ in the class of all  precompact Abelian groups.
\end{lemma}

\begin{proof}
Evidently, (1) implies (2) and  (2) implies (3). Let $\uuu$ be a $TB$-sequence in  $\langle\uuu\rangle$. 
Then $(\langle\uuu\rangle, \Tub(\langle\uuu\rangle))^\wedge$ separates the points of $\langle\uuu\rangle$. Let $\tau$ be the topology on $G$ whose base is all translates of $\Tub(\langle\uuu\rangle)$-open sets. Then $(\langle\uuu\rangle, \Tub(\langle\uuu\rangle))$ is an open subgroup of $(G,\tau)$. It is easy to see that $(G,\tau)^\wedge$ separates the points of $G$. Since $\uuu$ converges to zero in $\tau$, it is also converges to zero in $\tau^+$, where $\tau^+$ is the Bohr topology of $(G,\tau)$. Thus $\uuu$ is a $TB$-sequence in $G$. So (3) implies (1).

The last assertion follows from Proposition 1.8 and Lemma 3.6 in \cite{DGT}.
\end{proof}

For a sequence $\uuu =\{ u_n \}_{n\in \omega}$ of characters of a compact Abelian group $X$ set
\[
K_\mathbf{u} =\bigcap_{n\in \omega} \ker(u_n).
\]
The following  assertions is proved in \cite{DG}:
\begin{fact} {\rm \cite[Lemma 2.2(i)]{DG}} \label{fL}
For every sequence $\uuu =\{ u_n \}_{n\in \omega}$ of characters of a compact Abelian group $X$, the subgroup  $K_\mathbf{u}$ is a closed $G_\delta$-subgroup of $X$ and $K_\mathbf{u}=\langle\uuu\rangle^\perp$.
\end{fact}

The next two lemmas are natural analogues of Lemmas 2.2(ii) and 2.6 of \cite{DG}.
\begin{lemma} \label{l13}
Let $X$ be a compact Abelian groups and $\mathbf{u} =\{ u_n \}_{n\in \omega}$ be  a $T$-sequence in  $\widehat{X}$. Then  $s_\uuu (X)/K_\uuu$ is a $T$-characterized subgroup of $X/K_\uuu$.
\end{lemma}

\begin{proof}
Set $H:=s_\uuu (X)$ and $K:= K_\mathbf{u}$.   Let  $q : X \to X/K$ be  the quotient map. Then the  adjoint homomorphism $q^\wedge$ is an isomorphism from $(X/K)^\wedge$ onto $K^\perp$ in $X^\wedge$. For every $n\in\omega$, define the character $\widetilde{u}_n$ of $X/K$ as follows: $(\widetilde{u}_n, q(x))=(u_n, x)$ ($\widetilde{u}_n$ is well-defined since $K \subseteq \ker(u_n)$). Then $\widetilde{\mathbf{u}} =\{ \widetilde{u}_n \}_{n\in \omega}$ is a sequence of characters of $X/K$ such that $q^\wedge ( \widetilde{u}_n)= u_n$. Since $\uuu \subset K^\perp$, $\uuu$ is a $T$-sequence in $K^\perp$ by Lemma \ref{l11}. Hence $\widetilde{\uuu}$ is a $T$-sequence in $(X/K)^\wedge$ because $q^\wedge$ is an isomorphism.

We claim that $H/K = s_{\widetilde{\mathbf{u}}} (X/K)$. Indeed, for
every $h+K\in H/K$, by definition, we have
$(\widetilde{u}_n , h+K)=(u_n, h)\to 1.$
Thus $H/K \subseteq s_{\widetilde{\mathbf{u}}} (X/K)$. If $x+K\in s_{\widetilde{\mathbf{u}}} (X/K)$, then
$(\widetilde{u}_n , x+K)=(u_n, x)\to 1.$ This yields $x\in H$. Thus $x+K\in H/K$.
\end{proof}

Let $\uuu =\{ u_n\}_{n\in\omega}$ be a $T$-sequence in an Abelian group $G$. For every natural number $m$ set $\uuu_m =\{ u_n\}_{n\geq m}$. Clearly, $\uuu_m$ is a $T$-sequence in $G$,  $\Tu =\tau_{\mathbf{u}_m}$ and $s_\uuu (X)=s_{\uuu_m} (X)$ for every natural number $m$.
\begin{lemma} \label{l14}
Let $K$ be a closed subgroup of a compact Abelian group $X$ and $q: X\to X/K$ be the quotient map.
Then  $\widetilde{H}$ is a $T$-characterized subgroup of $X/K$ if and only if  $q^{-1}(\widetilde{H})$ is a $T$-characterized subgroup of $X$.
\end{lemma}

\begin{proof}
Let $\widetilde{H}$ be a $T$-characterized subgroup of $X/K$ and let a $T$-sequence $\widetilde{\mathbf{u}} =\{ \widetilde{u}_n \}_{n\in \omega}$ characterize $\widetilde{H}$. Set $H:=q^{-1}(\widetilde{H})$. We have to show that $H$ is a $T$-characterized subgroup of $X$.

Note that the adjoint homomorphism $q^\wedge$ is an isomorphism from $(X/K)^\wedge$ onto $K^\perp$ in $X^\wedge$. Set $\uuu =\{ u_n \}_{n\in\omega}$, where $u_n = q^\wedge (\widetilde{u}_n)$. Since $q^\wedge$ is injective, $\uuu$ is a $T$-sequence in $K^\perp$. By Lemma \ref{l11}, $\uuu$ is a $T$-sequence in $\widehat{X}$. So it is enough to show that $H=s_\uuu (X)$. This follows from the following
chain of equivalences. By definition, $x\in s_\uuu (X)$ if and only if
\[
(u_n, x)\to 1  \ \Leftrightarrow  \ (\widetilde{u}_n, q(x))\to 1\   \Leftrightarrow   \ q(x)\in \widetilde{H}= H/K \   \Leftrightarrow   \  x \in H.
\]
The last equivalence is due to the inclusion $K\subseteq H$.

Conversely, let $H:= q^{-1}(\widetilde{H})$ be a $T$-characterized subgroup of $X$ and a $T$-sequence  $\uuu =\{ u_n \}_{n\in\omega}$ characterize $H$.   Proposition 2.5 of \cite{DG} implies that we can find $m\in \NN$ such that $K\subseteq K_{\uuu_m}$. So, taking into account that $H=s_\uuu (X)=s_{\uuu_m} (X)$ for every natural number $m$, without loss of generality we can assume that $K\subseteq K_\uuu$. By Lemma \ref{l13}, $H/K_\uuu$ is a $T$-characterized subgroup of $X/K_\uuu$. Denote by $q_u$ the quotient homomorphism from $X/K$ onto $X/K_\uuu$. Then $\widetilde{H}=q_u^{-1} (H/K_\uuu)$ is $T$-characterized in $X/K$ by the previous paragraph of the proof.
\end{proof}

The next theorem is an analogue of Theorem B of \cite{DG}, and it reduces the study of $T$-characterized subgroups of compact Abelian groups to the study of $T$-characterized ones of compact Abelian metrizable groups:
\begin{theorem} \label{Char-T}
A subgroup $H$ of a compact Abelian group $X$ is $T$-characterized if and only if $H$ contains a closed $G_\delta$-subgroup $K$ of $X$ such that $H/K$ is a $T$-characterized subgroup of the compact metrizable group $X/K$.
\end{theorem}

\begin{proof}
Let $H$ be $T$-characterized in $X$  by a $T$-sequence $\mathbf{u} =\{ u_n \}_{n\in \omega}$ in $\widehat{X}$. Set $K:=
K_\mathbf{u}$.   Since $K$ is a closed $G_\delta$-subgroup of $X$ by Fact \ref{fL}, $X/K$ is
metrizable. By  Lemma \ref{l13}, $H/K$ is a $T$-characterized subgroup of $X/K$.

Conversely, let $H$ contain a closed $G_\delta$-subgroup $K$ of $X$ such that $H/K$ is a $T$-characterized
subgroup of the compact metrizable group $X/K$. Then $H$ is a $T$-characterized subgroup of $X$ by Lemma \ref{l14}.
\end{proof}

As it was noticed in \cite{Ga2} before Definition 2.33, for every $T$-sequence $\mathbf{u}$  in an infinite Abelian group  $G$  the subgroup $\langle\mathbf{u}\rangle$ is open in $(G, \Tu)$ (see also Lemma \ref{l11}), and hence, by Lemmas 1.4 and 2.2 of \cite{BCM}, the following sequences are exact:
\begin{equation} \label{1-1}
\begin{split}
& 0 \rightarrow (\langle\mathbf{u}\rangle, \Tu) \rightarrow (G, \Tu) \rightarrow G/\langle\mathbf{u}\rangle \rightarrow 0, \\
& 0 \rightarrow \left( G/\langle\mathbf{u}\rangle \right)^\wedge \rightarrow (G, \Tu)^\wedge \rightarrow (\langle\mathbf{u}\rangle, \Tu|_{\langle\mathbf{u}\rangle})^\wedge \rightarrow 0,
\end{split}
\end{equation}
where $\left( G/\langle\mathbf{u}\rangle \right)^\wedge \cong \langle\mathbf{u}\rangle^\perp$ is a compact subgroup of $(G, \Tu)^\wedge $ and $(\langle\mathbf{u}\rangle, \Tu)^\wedge \cong (G, \Tu)^\wedge /\langle\mathbf{u}\rangle^\perp$.

Let $\uuu =\{ u_n\}_{n\in\omega}$ be a $T$-sequence in an Abelian group $G$.
It is known  \cite{ZP2} that $\Tu$ is sequential, and hence $(G,\Tu)$ is a $k$-space.  So the natural homomorphism $\alpha :=\alpha_{(G,\Tu)} :(G,\Tu)\to (G,\Tu)^{\wedge\wedge}$ is continuous by \cite[5.12]{Aus}. Let us recall that $(G,\Tu)$  is MinAP if and only if $(G,\Tu)=\ker(\alpha)$.

To prove Theorem \ref{t1} we need the following:
\begin{fact}  {\em \cite{Ga}} \label{f3}
For  each $T$-sequence $\uuu$ in  a countably infinite Abelian group $G$ the group $(G,\Tu)^\wedge$ is Polish.
\end{fact}

Now we are in position to prove Theorem \ref{t1}.

{\em Proof of Theorem {\rm \ref{t1}}.} $(1)\Rightarrow (2)$
Let $H$ be a proper closed $T$-characterized subgroup of $X$ and a $T$-sequence $\uuu =\{ u_n\}_{n\in\omega}$ characterize $H$. Since $H$ is also characterized it is  a $G_\delta$-subgroup of $X$ by Fact \ref{f2}.  We have to show that $H^\perp$ admits a MinAP group topology.

Our idea of the proof is the following.  Set $G:= \widehat{X}$. By Fact \ref{f4}, $H^\perp$ is the von Neumann radical of $(G,\Tu)$. Now assume that  we found another $T$-sequence $\vvv$ which  characterizes   $H$ and  such that $\langle\vvv\rangle =H^\perp$ (maybe $\vvv =\uuu$).
By Fact  \ref{f4}, we have $\mathbf{n}(G,\Tv)=H^\perp =\langle\vvv\rangle$. Lemma \ref{l11} implies that the subgroup $(\langle\vvv\rangle, \Tv|_{\langle\vvv\rangle})$ of $(G,\Tv)$ is open, and hence it is dually closed and dually embedded in $(G,\Tv)$. Hence $\mathbf{n}(\langle\vvv\rangle, \Tv|_{\langle\vvv\rangle})=\mathbf{n}(G,\Tv)(=\langle\vvv\rangle)$ by Lemma 4 of \cite{Ga}. So $(\langle\vvv\rangle, \Tv|_{\langle\vvv\rangle})$  is MinAP. Thus $H^\perp =\langle\vvv\rangle$ admits a MinAP group topology, as desired.


We  find such a $T$-sequence $\vvv$ in 4 steps (in fact we show that $\vvv$ has the form $\uuu_m$ for some $m\in\mathbb{N}$).

{\it Step } 1. Let $q: X\to X/K_\uuu$ be the quotient map. For every $n\in\omega$,   define the character $\widetilde{u}_n$ of $X/K_\uuu$ by the equality $u_n =\widetilde{u}_n \circ q$ (this is possible since $K_\uuu \subseteq \ker(u_n)$). As it was shown in the proof of Lemma \ref{l13}, the sequence $\widetilde{\uuu} =\{ \widetilde{u}_n\}_{n\in\omega}$ is a $T$-sequence which characterizes  $H/K_\uuu$ in $X/K_\uuu$. Set $\widetilde{X}:=X/K_\uuu$ and $\widetilde{H} :=H/K_\uuu$. So that $\widetilde{H}=s_{\widetilde{\uuu}} (\widetilde{X})$. By \cite[5.34 and 24.11]{HR1} and since $K_\uuu \subseteq H$, we have
\begin{equation} \label{1-2}
H^\perp \cong (X/H)^\wedge \cong \left( \widetilde{X} / \widetilde{H} \right)^\wedge \cong \widetilde{H}^\perp .
\end{equation}
By Fact \ref{f2},  $\widetilde{X}$ is metrizable. Hence  $\widetilde{H}$ is also compact and metrizable, and $\widetilde{G}:= \widehat{\widetilde{X}}$ is a countable Abelian group by \cite[24.15]{HR1}. Since $H$ is a proper closed subgroup of $X$, (\ref{1-2}) implies that $\widetilde{G}$ is non-zero.

We claim that $\widetilde{G}$ is countably infinite. Indeed, suppose for a contradiction that $\widetilde{G}$ is finite. Then $X/K_\uuu =\widetilde{X}$ is also finite. Now  Fact \ref{fL} implies that $\langle\uuu \rangle$ is a finite subgroup of $G$. Since $\uuu$ is a $T$-sequence, $\uuu$ must be eventually equal to zero. Hence $H=s_\uuu (X)=X$ is not a proper subgroup of $X$, a contradiction.

{\it Step } 2. We claim that there is a natural number $m$ such that the group $(\langle\widetilde{\uuu}_m \rangle, \tau_{\widetilde{\mathbf{u}}}|_{\langle\widetilde{\uuu}_m \rangle})= (\langle\widetilde{\uuu}_m \rangle, \tau_{\widetilde{\mathbf{u}}_m}|_{\langle\widetilde{\uuu}_m \rangle})$ is MinAP.

Indeed, since $\widetilde{G}$ is countably infinite, we can apply Fact \ref{f4}. So $\widetilde{H}=(\widetilde{G}, \tau_{\widetilde{\mathbf{u}}})^\wedge$ algebraically. Since $\widetilde{H}$ and $(\widetilde{G}, \tau_{\widetilde{\mathbf{u}}})^\wedge$ are  Polish groups (see Fact \ref{f3}), $\widetilde{H}$ and $(\widetilde{G}, \tau_{\widetilde{\mathbf{u}}})^\wedge$ are topologically isomorphic by the uniqueness of the Polish group topology. Hence $(\widetilde{G}, \tau_{\widetilde{\mathbf{u}}})^{\wedge\wedge}=\widetilde{H}^\wedge$ is discrete. As it was noticed before the proof,  the natural homomorphism $\widetilde{\alpha}: (\widetilde{G}, \tau_{\widetilde{\mathbf{u}}}) \to (\widetilde{G}, \tau_{\widetilde{\mathbf{u}}})^{\wedge\wedge}$ is continuous. Since $(\widetilde{G}, \tau_{\widetilde{\mathbf{u}}})^{\wedge\wedge}$ is discrete we obtain that the von Neumann radical $\ker(\widetilde{\alpha})$ of $(\widetilde{G}, \tau_{\widetilde{\mathbf{u}}})$ is open in $\tau_{\widetilde{\mathbf{u}}}$. So there exists a natural number $m$ such that $\widetilde{u}_n \in \ker(\widetilde{\alpha})$ for every $n\geq m$.   Hence $\langle\widetilde{\uuu}_m \rangle \subseteq\ker(\widetilde{\alpha})$.
Lemma \ref{l11} implies that the subgroup $\langle\widetilde{\uuu}_m \rangle $ is open in $(\widetilde{G},\tau_{\widetilde{\mathbf{u}}})$, and hence it is dually closed and dually embedded in $(\widetilde{G}, \tau_{\widetilde{\mathbf{u}}})$. Now Lemma 4 of \cite{Ga} yields $\langle\widetilde{\uuu}_m \rangle =\ker(\widetilde{\alpha})$ and $(\langle\widetilde{\uuu}_m \rangle, \tau_{\widetilde{\mathbf{u}}}|_{\langle\widetilde{\uuu}_m \rangle})$ is MinAP.


{\it Step } 3. Set $\vvv =\{ v_n\}_{n\in\omega}$, where $v_n =u_{n+m}$ for every $n\in\omega$. Clearly, $\vvv$ is a $T$-sequence in $G$ characterizing $H$, $\Tu =\Tv$ and $K_\uuu \subseteq K_\vvv$. Let $t: X\to X/K_\vvv$ and $r: X/K_\uuu \to X/K_\vvv$ be the quotient maps. Analogously to Step 1 and the proof of Lemma \ref{l13},  the sequence $\widetilde{\vvv} =\{ \widetilde{v}_n\}_{n\in\omega}$ is a $T$-sequence in $\widehat{X/K_\vvv}$ which characterizes  $H/K_\vvv$ in $X/K_\vvv$, where $v_n =\widetilde{v}_n \circ t$. Since $t=r\circ q$ we have
\[
v_n =\widetilde{v}_n \circ t = t^\wedge (\widetilde{v}_n)= q^\wedge \left( r^\wedge (\widetilde{v}_n)\right),
\]
where $t^\wedge$, $r^\wedge$ and $q^\wedge$ are the adjoint homomorphisms to $t$, $r$ and $q$ respectively.

Since $q^\wedge$ and $r^\wedge$ are embeddings, we have $r^\wedge (\widetilde{v}_n) =\widetilde{u}_{n+m}$. In particular, $\langle \vvv\rangle \cong\langle \widetilde{\vvv}\rangle \cong \langle \widetilde{\uuu}_m \rangle \mbox{ and }$
\[
(\langle\widetilde{\uuu}_m \rangle, \tau_{\widetilde{\mathbf{u}}}|_{\langle\widetilde{\uuu}_m \rangle}) = (\langle\widetilde{\uuu}_m \rangle, \tau_{\widetilde{\mathbf{u}}_m}|_{\langle\widetilde{\uuu}_m \rangle})\cong (\langle\widetilde{\vvv} \rangle, \tau_{\widetilde{\mathbf{v}}}|_{\langle\widetilde{\mathbf{v}}\rangle}) \cong (\langle \vvv\rangle, \Tv|_{\langle \vvv\rangle}).
\]
By Step 2 $(\langle\widetilde{\uuu}_m \rangle, \tau_{\widetilde{\mathbf{u}}_m}|_{\langle\widetilde{\uuu}_m \rangle})$ is MinAP. Hence
$(\langle \vvv\rangle, \Tv|_{\langle \vvv\rangle})$ is MinAP as well.

{\it Step} 4. By the second exact sequence in (\ref{1-1}) applying to $\vvv$,   Fact \ref{f4} and since $(\langle\vvv \rangle, \tau_{\mathbf{v}}|_{\langle \vvv\rangle})$ is MinAP (by Step 3), we have $H=s_\vvv (X) = (G, \Tv)^\wedge = \left( G/\langle\mathbf{v}\rangle \right)^\wedge =\langle\mathbf{v}\rangle^\perp$ algebraically.  Thus $H^\perp =\langle\mathbf{v}\rangle $, and hence $H^\perp$ admits a  MinAP group topology generated by the $T$-sequence $\vvv$.

$(2)\Rightarrow (1)$: Since $H$ is a $G_\delta$-subgroup of $X$, $H$ is closed by \cite[Proposition 2.4]{DG} and $X/H$ is metrizable (due to the well known fact that a compact group of countable pseudocharacter is metrizable). Hence $H^\perp = (X/H)^\wedge$ is countable. Since $H^\perp$ admits a MinAP group topology, $H^\perp$ must be countably infinite. By Theorem 3.8  of \cite{Ga4}, $H^\perp$ admits a MinAP group topology generated by a $T$-sequence $\widetilde{\uuu} =\{ \widetilde{u}_n\}_{n\in\omega}$. By Fact \ref{f4}, this means that $s_{\widetilde{\uuu}} (X/H)= \{ 0\}$. Let $q: X\to X/H$ be the quotient map. Set $u_n =\widetilde{u}_n \circ q = q^\wedge  (\widetilde{u}_n)$. Since $q^\wedge$ is injective, $\uuu$ is a $T$-sequence in $\widehat{X}$ by Lemma \ref{l11}. We have to show that $H=s_\uuu (X)$. By definition,  $x\in s_\uuu (X)$ if and only if
\[
(u_n, x)= (\widetilde{u}_n, q(x)) \to 1 \Leftrightarrow q(x)\in s_{\widetilde{\uuu}} (X/H) \Leftrightarrow
q(x)=0 \Leftrightarrow  x\in H.
\]

(2)$\Leftrightarrow$(3) follows from  Theorem 3.8  of \cite{Ga4}. The theorem is proved.
$\Box$

{\em Proof of Theorem {\rm \ref{t2}}}.
$(1)\Rightarrow (2)$: Suppose for a contradiction that $X$ is not connected. Then, by \cite[24.25]{HR1}, the dual group $G=X^\wedge$ has a non-zero element $g$ of finite order. Then the subgroup $H:= \langle g\rangle^\perp$ of $X$ has finite index. Hence $H$ is an open subgroup of $X$. Thus $H$ is not $T$-characterized by Corollary \ref{c2}. This contradiction shows that $X$ must be connected.

$(2)\Rightarrow (1)$: Let $H$ be a proper $G_\delta$-subgroup of $X$. Then $H$ is closed by \cite[Proposition 2.4]{DG}, and $X/H$ is connected and non-zero. Hence $H^\perp \cong (X/H)^\wedge$ is countably  infinite and torsion free by \cite[24.25]{HR1}. Thus $H^\perp$   has infinite exponent. Therefore, by Theorem \ref{t1}, $H$ is $T$-characterized.
$\Box$

The next proposition is a simple corollary of Theorem B in \cite{DG}.
\begin{proposition}
The closure ${\bar H}$ of a characterized (in particular, $T$-characterized) subgroup $H$ of a compact Abelian group $X$ is a characterized subgroup of $X$.
\end{proposition}

\begin{proof}
By Theorem B of \cite{DG}, $H$ contains a compact $G_\delta$-subgroup $K$ of $X$. Then ${\bar H}$ is also a $G_\delta$-subgroup of $X$. Thus ${\bar H}$ is a characterized subgroup of $X$ by Theorem B of \cite{DG}.
\end{proof}
In general we cannot assert that the closure ${\bar H}$ of a  $T$-characterized subgroup $H$ of a compact Abelian group $X$ is also  $T$-characterized as the next example shows.
\begin{example} \label{exa1}
{\rm
Let $X =\ZZ(2)\times \TT$ and $G=\widehat{X}=\ZZ(2) \times \ZZ$ . It is known (see the end of $\mathbf{(1)}$ in \cite{Ga5}) that there is a $T$-sequence $\uuu$ in $G$ such that the von Neumann radical $\mathbf{n}(G,\Tu)$ of $(G,\Tu)$ is $\ZZ(2)\times\{ 0\}$, the subgroup $H:= s_\uuu (X)$ is countable and $\bar{H}=\{ 0\}\times \TT$. So the closure  $\bar{H}$ of the  countable $T$-characterized subgroup $H$ of $X$ is open. Thus $\bar{H}$ is not $T$-characterized by Corollary \ref{c2}.}
\end{example}

We do not know  answers to the following questions:
\begin{problem}
Let $H$ be a characterized subgroup of a compact Abelian group $X$ such that its closure $\bar{H}$ is $T$-characterized. Is  $H$ a $T$-characterized subgroup of $X$?
\end{problem}

\begin{problem}
Does there exists a metrizable Abelian compact group which has a countable non-$T$-characterized subgroup?
\end{problem}

\section{The Proofs of Theorems \ref{t3} and \ref{tS}} \label{s2}

Recall that a Borel subgroup $H$ of a Polish group $X$ is called {\it polishable} if there exists a Polish group topology $\tau$ on $H$ such that the inclusion map $i : (H, \tau ) \to X$ is continuous. Let $H$ be a $T$-characterized subgroup of a compact metrizable Abelian group $X$ by a $T$-sequence $\uuu =\{ u_n\}_{n\in\omega}$. Then, by  \cite[Theorem 1]{Ga}, $H$ is polishable by the  metric
\begin{equation} \label{01}
\rho (x,y) = d(x,y) + \sup \{ |(u_n, x) -(u_n, y)|, \; n\in\omega \},
\end{equation}
where $d$ is the initial metric on $X$.  Clearly, the topology generated by the metric $\rho$ on $H$ is finer than the induced one from $X$.

To prove Theorem \ref{t3} we need the following three lemmas.

For a real number $x$ we write $[x]$ for the integral part of  $x$  and $\| x\|$ for the distance from $x$ to the nearest integer. We also use the following inequality proved in \cite{Gab}
\begin{equation} \label{02}
\pi |\varphi | \leq | 1- e^{2\pi i\varphi} |\leq  2\pi |\varphi | , \quad \varphi\in \left[-\frac{1}{2} , \frac{1}{2} \right).
\end{equation}

\begin{lemma} \label{l1}
Let $\{ a_n\}_{n\in\omega} \subset \mathbb{N}$ be such that $a_n \to\infty$ and $a_n\geq 2, n\in\omega$. Set $u_n = \prod_{k\leq n} a_n$ for every $n\in\omega$. Then $\uuu =\{ u_n\}_{n\in\omega}$ is a $T$-sequence in $X=\mathbb{T}$, and the $T$-characterized subgroup $H=s_\uuu (\TT)$ of $\TT$ is a dense non-$F_\sigma$-subset of $\TT$.
\end{lemma}

\begin{proof}
We consider the circle group $\mathbb{T}$ as $\mathbb{R}/\mathbb{Z}$ and write it additively. So that $d(0,x)=\| x\|$ for every $x\in\mathbb{T}$. Recall (see, for example, the proof of Lemma 1 in \cite{Gab}) that every $x\in \TT$ has the unique representation in the form
\begin{equation} \label{1}
x= \sum_{n=0}^\infty \frac{c_n}{u_n},
\end{equation}
where $0\leq c_n < a_n$ and $c_n \not= a_n-1$ for infinitely many indices $n$.

It is known \cite{AaN} (see also (12) in the proof of Lemma 1 of \cite{Gab}) that $x$ with representation (\ref{1}) belongs to $H$ if and only if
\begin{equation} \label{2}
\lim_{n\to\infty} \frac{c_n}{a_n}(\mathrm{mod}\ 1) = 0.
\end{equation}
Hence $H$ is a dense subgroup of $\TT$. Thus $\uuu$  is even a $TB$-sequence in $\ZZ$ by Fact \ref{f1}.

We have to show that $H$ is not an $F_\sigma$-subset of $\TT$. Suppose for a contradiction that $H$ is   an $F_\sigma$-subset of $\TT$. Then $H=\cup_{n\in\NN} F_n$, where $F_n$ is a compact subset of $\TT$ for every $n\in \NN$. Since $H$ is a subgroup of $\TT$, without loss of generality we can assume that $F_n - F_n \subseteq F_{n+1}$. Since all  $ F_n$ are closed in $(H, \rho)$ as well,  the Baire theorem  implies that there are $0<\varepsilon <0.1$ and $m\in \NN$ such that
$F_m \supseteq \{ x: \rho(0,x) \leq \varepsilon\}$.

Fix arbitrarily $l>0$ such that $\frac{2}{u_{l-1}} <\frac{\varepsilon}{20}$. For every natural number $k>l$, set
\[
x_k := \sum_{n=l}^k \frac{1}{u_n} \cdot \left[ \frac{(a_n -1)\varepsilon}{20} \right].
\]
Then, for every $k>l$, we have
\[
 x_k = \sum_{n=l}^k \frac{1}{u_n} \cdot \left[ \frac{(a_n -1)\varepsilon}{20} \right]  < \sum_{n=l}^k \frac{1}{u_{n-1}} \cdot  \frac{\varepsilon}{20} < \frac{1}{u_{l-1}} \sum_{n=0}^{k-l} \frac{1}{2^n} < \frac{2}{u_{l-1}} <\frac{\varepsilon}{20}< \frac{1}{2}.
\]
This inequality and (\ref{02}) imply that
\begin{equation} \label{11}
d(0, x_k) = \| x_k\| = x_k  <\frac{\varepsilon}{20}, \; \mbox{ for every } k>l.
\end{equation}

For every  $s\in\omega$ and every natural number $k>l$, we estimate $| 1-(u_{s},x_k)|$ as follows.

{\it Case} 1.  Let $s<k$. Set $q=\max\{ s+1, l\}$. By the definition of $x_k$,  we have
\[
\begin{split}
2\pi  \left[ (u_s \cdot x_k)\right.  \left. (\mathrm{mod}\ 1) \right] & =  2\pi \left[ u_s \sum_{n=l}^k \frac{1}{u_n} \cdot \left[ \frac{(a_n -1)\varepsilon}{20} \right] (\mathrm{mod}\ 1)\right]  < 2\pi \sum_{n=q}^k \frac{u_s}{u_n} \cdot  \frac{(a_n -1)\varepsilon}{20} \\
& < \frac{\pi\varepsilon}{10}\left( 1+ \frac{1}{a_{s+1}} +\frac{1}{a_{s+1}a_{s+2}}+\frac{1}{a_{s+1}a_{s+2}a_{s+3}}+ \dots \right) \\
& < \frac{\pi\varepsilon}{10}\left( 1+ \frac{1}{2} + \frac{1}{2^2} +\frac{1}{2^3} +\dots \right)=\frac{\pi\varepsilon}{10}\cdot 2<  \frac{2\varepsilon}{3} <\frac{1}{2}.
\end{split}
\]
This inequality and (\ref{02}) imply
\begin{equation} \label{12}
|1- (u_s,  x_k)|= \left| 1-  \exp\left\{ 2\pi i\cdot\left[ (u_s \cdot x_k)\, (\mathrm{mod}\ 1) \right] \right\} \right| < \frac{2\varepsilon}{3}.
\end{equation}

{\it Case} 2. Let $s\geq k$. By the definition of $x_k$,  we have
\begin{equation} \label{13}
| 1-(u_{s},x_k)| =0.
\end{equation}
In particular, (\ref{13}) implies that $x_k \in H$ for every $k>l$.

Now,  for every $k>l$, (\ref{01}) and (\ref{11})-(\ref{13}) imply
\[
\rho(0,x_k) <  \frac{\varepsilon}{20} + \frac{2\varepsilon}{3}< \varepsilon.
\]
Thus $x_k \in F_m$ for every natural number $k> l$.
Clearly,
\[
x_k \to x:=\sum_{n=l}^\infty \frac{1}{u_n} \cdot \left[ \frac{(a_n -1)\varepsilon}{20} \right] \quad \mbox{ in } \TT.
\]
Since $F_m$ is a compact subset of $\TT$, we have  $x\in F_m$. Hence $x \in H$. On the other hand,  we have
\[
\lim_{n\to\infty} \frac{1}{a_n} \cdot \left[ \frac{(a_n -1)\varepsilon}{20} \right] (\mathrm{mod}\ 1) = \frac{\varepsilon}{20} \not= 0.
\]
So  (\ref{2}) implies that $x\not\in H$. This contradiction shows that $H=s_\uuu (\TT)$  is not an $F_\sigma$-subset of $\TT$.
\end{proof}

For a prime number $p$, the group $\mathbb{Z}(p^\infty)$ is regarded as the collection of fractions $m/p^n \in [0,1)$. Let $\Delta_p$ be the compact group of $p$-adic integers. It is well known that $\widehat{\Delta_p}=\mathbb{Z}(p^\infty)$.
\begin{lemma} \label{l2}
Let $X=\Delta_p$. For an increasing sequence of natural numbers $0< n_0<n_1<\dots $ such that $ n_{k+1} -n_k \to\infty$, set
\[
u_k =\frac{1}{p^{n_k +1}} \in \mathbb{Z}(p^\infty).
\]
Then the sequence $\uuu =\{ u_k\}_{k\in\omega}$ is a $T$-sequence in $\mathbb{Z}(p^\infty)$, and the $T$-characterized subgroup $H=s_\uuu (\Delta_p)$ is a dense non-$F_\sigma$-subset of $\Delta_p$.
\end{lemma}

\begin{proof}
Let $\omega=(a_n)_{n\in\omega} \in \Delta_p$, where $0\leq a_n <p$ for every  $n\in\omega$. Recall that, for every $k\in\omega$,  \cite[25.2]{HR1} implies
\begin{equation} \label{3}
(u_k,\omega) =\exp \left\{ \frac{2\pi i}{p^{n_k +1}} \left( a_0 + pa_1 +\dots + p^{n_k} a_{n_k} \right)\right\} .
\end{equation}
Further, by  \cite[10.4]{HR1}, if $\omega \not=0$, then
$d(0 ,\omega) = 2^{-n}$, where  $n$ is the minimal index such that $a_n \not= 0$.

Following \cite[2.2]{Ga1}, for every $\omega=(a_n) \in \Delta_p$ and every natural number $k>1$, set
\[
m_k =m_k (\omega) =\max \{ j_k, n_{k-1} \},
\]
where
\[
j_k  = n_k \mbox{ if } 0< a_{n_k} <p-1,
\]
and otherwise
\[
j_k =\min\{ j  :  \mbox{ either } a_s =0 \mbox{ for } j<s\leq n_k, \mbox{or } a_s =p-1 \mbox{ for  } j<s\leq n_k \}.
\]

In \cite[2.2]{Ga1} it is shown that
\begin{equation} \label{4}
\omega \in s_{\mathbf{u}} (\Delta_p ) \mbox{ if and only if } n_k -m_k \to\infty.
\end{equation}
So $H:=s_{\mathbf{u}} (\Delta_p )$ contains  the identity $\mathbf{1} =(1,0,0,\dots)$ of $\Delta_p$. By \cite[Remark 10.6]{HR1}, $\langle \mathbf{1} \rangle$ is dense in $\Delta_p$.
Hence $H$ is dense in $\Delta_p$ as well. Now Fact \ref{f1} implies that $\uuu$ is a $T$-sequence in $\mathbb{Z}(p^\infty)$.

We have to show that $H$ is not  an $F_\sigma$-subset of $\Delta_p$.
Suppose for a contradiction that $H=\cup_{n\in\NN} F_n$ is an $F_\sigma$-subset of $\Delta_p$, where $F_n$ is a compact subset of $\Delta_p$ for every $n\in \NN$. Since $H$ is a subgroup of $\Delta_p$, without loss of generality we can assume that $F_n - F_n \subseteq F_{n+1}$. Since all  $ F_n$ are closed in $(H, \rho)$ as well,  the Baire theorem implies that there are $0<\varepsilon <0.1$ and $m\in \NN$ such that $F_m \supseteq \{ x: \rho(0,x) \leq \varepsilon\}$.

Fix a natural number $s$ such that $\frac{1}{2^{s}} < \frac{\varepsilon}{20}$. Choose a natural number $l> s$ such that, for every natural number $w\geq l$, we have
\begin{equation} \label{41}
n_{w+1} -n_w >s.
\end{equation}

For every  $r\in\NN$, set
\[
\omega_r :=(a_n^r), \mbox{ where } a_n^r = \left\{
\begin{split}
1, & \mbox{ if } n= n_{l+i} -s \mbox{ for some } 1\leq i \leq r,\\
0, & \mbox{ otherwise}.
\end{split}
\right.
\]
Then, for every  $r\in\NN$,  (\ref{41}) implies that $\omega_r$ is well-defined  and
\begin{equation} \label{411}
d(0,\omega_r ) =\frac{1}{2^{n_{l+1}-s}} < \frac{1}{2^{n_l}} \leq \frac{1}{2^{l}} <\frac{1}{2^{s}} <\frac{\varepsilon}{20}.
\end{equation}

Note that
\begin{equation} \label{42}
1+p+ \dots + p^k = \frac{p^{k+1} -1}{p-1} < p^{k+1}.
\end{equation}

For every  $k\in\omega$ and every $r\in\NN$, we estimate $| 1-(u_{k},\omega_r)|$ as follows.

{\it Case} 1. Let $k\leq l$. By   (\ref{3}), (\ref{41}) and the definition of $\omega_r$ we have
\begin{equation} \label{43}
| 1-(u_{k},\omega_r)| =0.
\end{equation}

{\it Case} 2.
Let $l< k\leq l+r$. Then (\ref{42}) yields
\[
\frac{2\pi }{p^{n_k +1}} \left|  p^{n_{l+1} -s} +\dots + p^{n_{k} -s} \right| < \frac{2\pi }{p^{n_k +1}} \cdot p^{n_k -s+1} =\frac{2\pi }{p^{s}}\leq \frac{2\pi }{2^{s}} < \frac{\varepsilon}{2} < \frac{1}{2}.
\]
This inequality and the inequalities (\ref{02}) and  (\ref{3}) imply
\begin{equation} \label{44}
| 1-(u_{k},\omega_r)|  =\left| 1- \exp \left\{ \frac{2\pi i}{p^{n_k +1}} \left(  p^{n_{l+1} -s} +\dots + p^{n_{k} -s} \right)\right\} \right| < \frac{\varepsilon}{2}.
\end{equation}

{\it Case} 3. Let $l+r < k$. By (\ref{42}) we have
\[
\begin{split}
\frac{2\pi }{p^{n_k +1}} \left|  p^{n_{l+1} -s} +\dots + p^{n_{l+r} -s} \right| &
 < \frac{2\pi }{p^{n_k +1}} \cdot p^{n_{l+r} -s+1} \\
 & < \frac{2\pi }{p^{n_k +1}} \cdot p^{n_k -s+1} =\frac{2\pi }{p^{s}}\leq \frac{2\pi }{2^{s}} < \frac{\varepsilon}{2}.
 \end{split}
\]
These inequalities, (\ref{02}) and  (\ref{3}) immediately yield
\begin{equation} \label{45}
| 1-(u_{k},\omega_r)|  =\left| 1- \exp \left\{ \frac{2\pi i}{p^{n_k +1}} \left(  p^{n_{l+1} -s} +\dots + p^{n_{l+r} -s} \right)\right\} \right|  < \frac{\varepsilon}{2},
\end{equation}
and
 \begin{equation} \label{451}
| 1-(u_{k},\omega_r)|  < \frac{2\pi }{p^{n_k +1}} \cdot p^{n_{l+r} -s+1} \to 0,\; \mbox{ as } k\to\infty.
\end{equation}
So, (\ref{451}) implies that $\omega_r \in H$ for every $r\in\NN$.

For every $r\in\NN$, by (\ref{01}), (\ref{411}) and (\ref{43})-(\ref{45}) we have
\[
\rho( 0, \omega_r)  = d(0,\omega_r ) + \sup\left\{ \left| 1 - (u_k, \omega_r )\right|,\; k \in\omega \right\}
 < \frac{\varepsilon}{20} + \frac{\varepsilon}{2} < \varepsilon .
\]
Thus $\omega_r \in F_m$ for every  $r\in\NN$. Evidently,
\[
\omega_r \to \widetilde{\omega}=(\widetilde{a}_n) \mbox{ in } \Delta_p, \mbox{ where }  \widetilde{a}_n = \left\{
\begin{split}
1, & \mbox{ if } n= n_{l+i} -s \mbox{ for some  } i\in\NN ,\\
0, & \mbox{ otherwise}.
\end{split}
\right.
\]
Since $F_m$ is a compact subset of $\Delta_p$, we have  $\widetilde{\omega}\in F_m$. Hence  $\widetilde{\omega}\in H$. On the other hand, it is clearly that $m_k(\widetilde{\omega})= n_k - s$ for every $k\geq l+1$. Thus for every $k\geq l+1$,  $n_k -m_k(\widetilde{\omega}) =s \not\to \infty$. Now (\ref{4}) implies that $\widetilde{\omega} \not\in H$. This contradiction shows that
$H$ is not  an $F_\sigma$-subset of $\Delta_p$.
\end{proof}

\begin{lemma} \label{l3}
Let  $X=\prod_{n\in\omega} \mathbb{Z} (b_n)$, where $1<b_0 < b_1 < \dots, $ and $G:=\widehat{X}=\bigoplus_{n\in\omega} \mathbb{Z} (b_n)$. Set $\mathbf{u}=\{ u_n \}_{n\in\omega}$, where $u_n =1 \in \mathbb{Z} (b_n)^\wedge \subset G$ for every $n\in\omega$.
Then $\uuu $ is a $T$-sequence in $G$, and the $T$-characterized subgroup $H=s_\uuu (X)$ is a dense non-$F_\sigma$-subset of $X$.
\end{lemma}

\begin{proof}
Set $H:= s_\mathbf{u} (X)$. In \cite[2.3]{Ga1} it is shown that
\begin{equation} \label{5}
\omega=(a_n) \in s_\mathbf{u} (X) \mbox{ if and only if } \left\| \frac{a_n}{b_n} \right\| \to 0.
\end{equation}
So $\bigoplus_{n\in\omega} \mathbb{Z} (b_n) \subseteq H$. Thus $H$ is dense in $X$. Now Fact \ref{f1} implies that $\uuu$ is a $T$-sequence in $G$.

We have to show that $H$ is not  an $F_\sigma$-subset of $X$.
Suppose for a contradiction that $H=\cup_{n\in\NN} F_n$ is an $F_\sigma$-subset of $X$, where $F_n$ is a compact subset of $X$ for every $n\in \NN$.  Since $H$ is a subgroup of $X$, without loss of generality we can assume that $F_n - F_n \subseteq F_{n+1}$. Since all  $ F_n$ are closed in $(H, \rho)$ as well, the Baire theorem yields  that there are $0<\varepsilon <0.1$ and $m\in \NN$ such that
$F_m \supseteq \{ \omega\in X: \rho(0,\omega) \leq \varepsilon\}$.

Note 
that $d(0,\omega)=2^{-l}$, where $0\not= \omega=(a_n)_{n\in\omega} \in X$ and $l$ is the minimal index such that $a_l \not= 0$. Choose $l$ such that $2^{-l} < \varepsilon/3$. For every natural number  $k>l$, set
\[
\omega_k :=(a_n^k), \mbox{ where } a_n^k = \left\{
\begin{split}
\left[ \frac{\varepsilon b_n}{20} \right], & \mbox{ for every } n \mbox{ such that } l\leq n \leq k,\\
0, & \mbox{ if either } 1\leq n < l \mbox{ or } k<n.
\end{split}
\right.
\]
Since $(u_n, \omega_k )=1$ for every $n>k$, we obtain that $\omega_k \in H$ for every $k>l$.
For every  $n\in\omega$ we have
\[
2\pi \cdot \frac{1}{b_n} \left[ \frac{\varepsilon b_n}{20} \right] < \frac{2\pi\varepsilon}{20} < \varepsilon <\frac{1}{2}.
\]
This inequality and the inequalities  (\ref{01}) and (\ref{02}) imply
\[
\begin{split}
\rho( 0, \omega_k) & = d(0,\omega_k ) + \sup\left\{ \left| 1 - (u_n, \omega_k )\right|,\; n \in\omega \right\} \\
 & \leq \frac{1}{2^l} + \max\left\{ \left| 1 - \exp\left\{ 2\pi i \frac{1}{b_n} \left[ \frac{\varepsilon b_n}{20} \right] \right\} \right|,\; l\leq n \leq k \right\} \\
& \leq \frac{\varepsilon}{3} + 2\pi \cdot \max\left\{ \frac{1}{b_n} \left[ \frac{\varepsilon b_n}{20} \right] ,\; l\leq n \leq k \right\} < \frac{\varepsilon}{3} + \frac{2\pi\varepsilon}{20} < \varepsilon .
\end{split}
\]
Thus $\omega_k \in F_m$ for every natural  number $k>l$. Evidently,
\[
\omega_k \to \widetilde{\omega}=\left( \widetilde{a}_n \right)_{n\in\omega}  \mbox{ in } X, \mbox{ where } \widetilde{a}_n = \left\{
\begin{split}
0, & \mbox{ if } 0\leq n < l ,\\
\left[ \frac{\varepsilon b_n}{20} \right], & \mbox{ if }  l\leq n .
\end{split}
\right.
\]
Since $F_m$ is a compact subset of $X$, we have  $\widetilde{\omega} \in F_m $. Hence $\widetilde{\omega}\in H$. On the other hand,  since $b_n \to \infty$ we have
\[
\lim_{n\to\infty} \left\| \frac{\widetilde{a}_n}{b_n} \right\| = \lim_{n\to\infty} \frac{1}{b_n} \left[ \frac{\varepsilon b_n}{20} \right] =\frac{\varepsilon}{20} \not= 0.
\]
Thus $\widetilde{\omega} \not\in H$ by (\ref{5}). This contradiction shows that $H$ is not  an $F_\sigma$-subset of $X$.
\end{proof}

Now we are in position to prove Theorems \ref{t3} and \ref{tS}.

{\em Proof of Theorem {\rm \ref{t3}}}.
Let $X$ be a compact Abelian group of infinite exponent. Then $G:= \widehat{X}$ also has infinite exponent. It is well-known that $G$ contains a countably infinite subgroup $S$ of one of the following form:
\begin{enumerate}
\item[(a)] $S\cong \ZZ$;
\item[(b)] $S\cong  \mathbb{Z} (p^\infty)$;
\item[(c)] $S\cong \bigoplus_{n\in\omega} \ZZ (b_n)$, where  $1<b_0 < b_1 <\dots$.
\end{enumerate}
Fix such a subgroup $S$. Set $K=S^\perp$ and $Y= X/K \cong S_d^\wedge$,
where $S_d$ denotes the group $S$ endowed with the discrete topology. Since $S$ is countable, $Y$ is metrizable. Hence $\{ 0\}$ is a $G_\delta$-subgroup of $Y$. Thus $K$ is a $G_\delta$-subgroup of $X$. Let $q: X\to Y$ be the quotient map.
By Lemmas \ref{l1}-\ref{l3},   the compact group $Y$ has a dense $T$-characterized subgroups $\widetilde{H}$ which is not an $F_\sigma$-subset of $Y$. Lemma  \ref{l14} implies that $H:=q^{-1}(\widetilde{H})$ is a dense $T$-characterized subgroups of $X$.
 Since the continuous image of  an $F_\sigma$-subset of a compact group is  an $F_\sigma$-subset as well, we obtain that  $H$ is  not an $F_\sigma$-subset of $X$.
Thus the subgroup $H$ of $X$ is   $T$-characterized but it is not an $F_\sigma$-subset of $X$.
 The theorem is proved.
$\Box$

{\em Proof of Theorem {\rm \ref{tS}}}.
(1) follows from Fact \ref{fBor}.

(2): By Lemma 3.6 in  \cite{DG}, every infinite compact Abelian group $X$ contains a dense characterized subgroup $H$. By Fact \ref{f1}, $H$ is $T$-characterized. Since every $G_\delta$-subgroup of $X$ is closed in $X$ by Proposition 2.4 of \cite{DG}, $H$ is not a $G_\delta$-subgroup of $X$.

(3) follows from Theorem \ref{t2} and the aforementioned Proposition 2.4 of \cite{DG}.

(4) follows from Fact \ref{fBor}.

(5) follows from Corollary \ref{c4}.
$\Box$

It is trivial that $\mathrm{Char}_T(X) \subseteq \mathrm{Char}(X)$ for every compact Abelian group $X$. For the circle group $\TT$ we have.

\begin{proposition} \label{pCharT}
$\mathrm{Char}_T (\TT) = \mathrm{Char} (\TT)$.
\end{proposition}

\begin{proof}
We have to show only that $\mathrm{Char}(\TT) \subseteq \mathrm{Char}_T (\TT)$. Let $H=s_\uuu (\TT) \in \mathrm{Char} (\TT)$ for some sequence $\uuu$ in $\ZZ$.

If $H$ is infinite, then $H$ is dense in $\TT$. So $\uuu$ is a $T$-sequence in $\mathbb{Z}$ by Fact \ref{f1}. Thus $H\in \mathrm{Char}_T (\TT)$.

If $H$ is finite, then $H$ is closed in $\TT$. Clearly, $H^\perp$ has infinite exponent. Thus $H\in \mathrm{Char}_T (\TT)$ by Theorem \ref{t1}.
\end{proof}
Note that,  if  a  compact Abelian group  $X$ satisfies the equality $\mathrm{Char}_T (X) = \mathrm{Char} (X)$, then $X$ is connected by Fact \ref{f2} and Theorem \ref{t2}. This fact and Proposition \ref{pCharT} justify the next problem:
\begin{problem} \label{prob41}
Does there exists a connected compact Abelian group  $X$ such that $\mathrm{Char}_T (X) \not= \mathrm{Char} (X)$? Is it true that $\mathrm{Char}_T (X) = \mathrm{Char} (X)$ if and only if $X$ is connected?
\end{problem}

For a compact Abelian group $X$, the set of all subgroups of $X$ which are both $F_{\sigma\delta}$- and $G_{\delta\sigma}$-subsets of $X$ we denote by $\mathrm{S}\Delta_3^0(X)$. To complete the study of the Borel hierarchy of ($T$-)characterized subgroups of $X$ we have to answer to the next question.
\begin{problem}
Describe compact Abelian groups  $X$ of infinite exponent for which $\mathrm{Char}(X)\subseteq \mathrm{S}\Delta_3^0(X)$. For which compact Abelian groups $X$ of infinite exponent there exists a $T$-characterized subgroup $H$ that does not belong to $\mathrm{S}\Delta_3^0(X)$?
\end{problem}

\section{$\mathfrak{g}_T$-closed and $\mathfrak{g}_T$-dense subgroups of compact Abelian groups}

The following closure operator $\mathfrak{g}$ of the category of Abelian topological groups is defined in \cite{DMT}.
Let $X$ be an Abelian topological group and $H$ its arbitrary subgroup. The closure operator $\mathfrak{g}=\mathfrak{g}_X$ is defined as follows
\[
\mathfrak{g}_X (H) := \bigcap_{ \mathbf{u} \in \widehat{X}^\NN } \left\{ s_{\mathbf{u}} (X) : \; H \leq s_{\mathbf{u}} (X) \right\},
\]
and we say that $H$ is $\mathfrak{g}$-{\it closed} if $H=\mathfrak{g} (H)$, and  $H$ is $\mathfrak{g}$-{\it dense} if $\mathfrak{g} (H)=X$.

The set of all $T$-sequences in the dual group $\widehat{X}$ of a compact Abelian group $X$ we denote by $\mathcal{T}_s (\widehat{X})$. Clearly, $\mathcal{T}_s (\widehat{X})\subsetneqq \widehat{X}^\NN$. Let $H$ be a subgroup of $X$. In analogy to the closure operator $\mathfrak{g}$, $\mathfrak{g}$-closure and $\mathfrak{g}$-density, the operator $\mathfrak{g}_T$ is defined as follows
\[
\mathfrak{g}_T (H) := \bigcap_{ \mathbf{u} \in \mathcal{T}_s (\widehat{X}) } \left\{ s_{\mathbf{u}} (X) : \; H \leq s_{\mathbf{u}} (X) \right\},
\]
and we say that $H$ is {\it $\mathfrak{g}_T$-closed} if $H=\mathfrak{g}_T (H)$, and  $H$ is {\it  $\mathfrak{g}_T$-dense} if $\mathfrak{g}_T (H)=X$.

In this section we study some properties of $\mathfrak{g}_T$-closed and $\mathfrak{g}_T$-dense subgroups of a compact Abelian group $X$.
Note that every $\mathfrak{g}$-dense subgroup of $X$ is dense by Lemma 2.12 of \cite{DMT}, but for  $\mathfrak{g}_T$-dense subgroups the situation changes:
\begin{proposition} \label{p41}
Let $X$ be  a compact Abelian group.
\begin{enumerate}
\item[{\rm (1)}] If $H$ is a $\mathfrak{g}_T$-dense subgroup of $X$, then the closure ${\bar H}$ of $H$ is an open subgroup of $X$.
\item[{\rm (2)}] Every  open subgroup of a compact Abelian group $X$ is $\mathfrak{g}_T$-dense.
\end{enumerate}
\end{proposition}

\begin{proof}
(1) Suppose for a contradiction that ${\bar H}$ is not open in $X$. Then $X/{\bar H}$ is an infinite compact group. By Lemma 3.6 of \cite{DG}, $X/{\bar H}$ has a proper dense characterized subgroup $S$. Fact \ref{f1} implies that $S$ is a $T$-characterized subgroup of $X/{\bar H}$. Let $q:X\to X/{\bar H}$ be the quotient map. Then Lemma \ref{l14} yields that $q^{-1}(S)$ is a  $T$-characterized dense subgroup of $X$ containing $H$. Since $q^{-1}(S)\not= X$, we obtain that $H$ is not $\mathfrak{g}_T$-dense in $X$, a contradiction.

(2) Let $H$ be an open subgroup of $X$. If $H=X$ the assertion is trivial. Assume that $H$ is a proper subgroup (so $X$ is disconnected).  Let $\uuu$ be an arbitrary $T$-sequence such that $H\subseteq s_\uuu (X)$. Since $H$ is open, $s_\uuu (X)$ is open as well. Now  Corollary \ref{c2} implies that $s_\uuu (X)=X$. Thus $H$ is $\mathfrak{g}_T$-dense in $X$.
\end{proof}

Proposition \ref{p41}(2) shows that $\mathfrak{g}_T$-density may essentially differ from the usual $\mathfrak{g}$-density. In the next theorem we characterize all compact Abelian groups for which all $\mathfrak{g}_T$-dense subgroups are also dense.

\begin{theorem} \label{t41}
All $\mathfrak{g}_T$-dense subgroups of a compact Abelian group $X$ are dense if and only if $X$ is connected.
\end{theorem}

\begin{proof}
Assume that all $\mathfrak{g}_T$-dense subgroup of $X$ are dense. Proposition \ref{p41}(2) implies that $X$ has no open proper subgroups. Thus $X$ is connected by \cite[7.9]{HR1}.

Conversely, let $X$ be connected and $H$ be a $\mathfrak{g}_T$-dense subgroup of $X$.  Proposition \ref{p41}(1) implies that the closure ${\bar H}$ of $H$ is an open subgroup of $X$. Since $X$ is connected we obtain that ${\bar H}= X$. Thus $H$ is dense in $X$.
\end{proof}

For  $\mathfrak{g}_T$-closed subgroups we have:
\begin{proposition} \label{p42}
Let $X$ be  a compact Abelian group.
\begin{enumerate}
\item[{\rm (1)}] Every proper  open subgroup $H$ of $X$ is a  $\mathfrak{g}$-closed  non-$\mathfrak{g}_T$-closed subgroup.
\item[{\rm (2)}] If every $\mathfrak{g}$-closed subgroup of $X$ is $\mathfrak{g}_T$-closed, then $X$ is connected.
\end{enumerate}
\end{proposition}

\begin{proof}
(1)  The subgroup $H$ is $\mathfrak{g}_T$-dense in $X$ by Proposition \ref{p41}. Therefore $H$ is  not $\mathfrak{g}_T$-closed. On the other hand,
$H$ is $\mathfrak{g}$-closed in $X$ by Theorem A of \cite{DG}.

(2) Item (1)  implies that $X$ has no open subgroups. Thus $X$ is connected by \cite[7.9]{HR1}.
\end{proof}

We do not know whether the converse in Proposition \ref{p42}(2) holds true:
\begin{problem} \label{prob42}
Let a  compact Abelian group $X$ be connected. Is it true that  every $\mathfrak{g}$-closed subgroup of $X$ is also  $\mathfrak{g}_T$-closed?
\end{problem}

{\bf Historical Note.} This paper (with $a_n =n$ in Lemma 3.1) was sent  for possible publications to the journal ``Topology Proceedings'' at 25 November 2012. However, the author till now did not received even a report from the referee. Since the paper is cited in \cite{DG,Ga4} and other articles which have already been published, the author decided to put it in ArXiv.


\bibliographystyle{plain}

\begin{thebibliography}{10}

\smallskip

\bibitem{AaN}
{ J. Aaronson, M. Nadkarni,} $L_{\infty}$ eigenvalues and $L_2$ spectra of non-singular transformations, Proc. London Math. Soc.  55 (3) (1987) 538-570.


\bibitem{Aus}
{L.~Au{\ss}enhofer}, Contributions to the duality theory of
Abelian topological groups and to the theory of nuclear groups,
Dissertation, T\"{u}bingen, 1998, Dissertationes Math. (Rozprawy
Mat.)  \textbf{384} (1999), 113p.


\bibitem{BCM}
{W.~Banaszczyk, M.~J.~Chasco, E.~Martin-Peinador}, Open subgroups
and Pontryagin duality,  Math. Z.  \textbf{215} (1994), 195--204.


\bibitem{BDM}
{G.~Barbieri, D.~Dikranjan, C.~Milan, H.~Weber,} Answer to
Raczkowski's question on convergent sequences of integers,
Topology Appl.  \textbf{132} (2003), 89--101.

\bibitem{BDMW}
{G.~Barbieri, D.~Dikranjan, C.~Milan, H.~Weber,} Convergent sequences in precompact group topologies, Appl. Gen. Topology   \textbf{6} (2005), 149--169.

\bibitem{BSW}
M.~Beiglb\" ock, C.~Steineder, R.~Winkler, Sequences and filters of characters characterizing subgroups of compact abelian groups, Topology Appl.  \textbf{153}  (2006), 1682--1695.


\bibitem{BDS}
{A.~Bir\'{o}, J.-M.~Deshouillers, V.~T.~S\'{o}s,} Good approximation and
characterization of subgroups of $\mathbb{R} /\mathbb{Z}$, Studia
Sci. Math. Hungar.   \textbf{38} (2001), 97--113.

\bibitem{Bor}
{J.-P. Borel}, Sur certains sous-groupes de $\mathbb{R}$ li\'{e}s \`{a} la suite des factorielles, Colloq. Math.  \textbf{62} (1991), 21--30.


\bibitem{CRT}
W.~Comfort, S.~Raczkowski, F.-J.~Trigos-Arrieta,  Making group topologies with,
and without, convergent sequences,  Appl. Gen.  Topology \textbf{7}  (2006), 109--124.





\bibitem{DG}
D.~Dikranjan, S.~Gabriyelyan,  On characterized subgroups of compact abelian groups, Topology Appl. \textbf{160} (2013), 2427--2442.


\bibitem{DGT}
D.~Dikranjan, S.~Gabriyelyan, V.~Tarieladze, Characterizing sequences for precompact group topologies, J. Math. Anal. Appl. \textbf{412} (2014), 505--519.




\bibitem{DiK}
{D.~Dikranjan, K.~Kunen,} Characterizing subgroups of compact abelian groups,  J. Pure Appl. Algebra  \textbf{208} (2007), 285--291.

\bibitem{DMT}
D.~Dikranjan, C.~Milan,  A.~Tonolo,   A characterization of the MAP abelian groups,   J. Pure Appl. Algebra  \textbf{197} (2005), 23--41.

\bibitem{Ga5}
{ S.~Gabriyelyan,} Characterization of almost maximally almost-periodic groups, Topology Appl. \textbf{156} (2009), 2214--2219.


\bibitem{Gab}
{ S.~Gabriyelyan}, Groups of quasi-invariance and the Pontryagin duality, Topology Appl.  \textbf{157} (2010), 2786--2802.


\bibitem{Ga}
{ S.~Gabriyelyan}, On $T$-sequences and characterized subgroups, Topology Appl. \textbf{157} (2010), 2834--2843.


\bibitem{Ga1}
S.~Gabriyelyan, Reflexive group topologies on Abelian groups, J. Group Theory \textbf{13} (2010), 891-901.


\bibitem{Ga2}
{ S.~Gabriyelyan}, Characterizable groups: some results and open questions, Topology Appl. \textbf{159} (2012), 2378--2391.

\bibitem{GaFin}
{ S.~Gabriyelyan,} Finitely generated subgroups as a von Neumann radical of an Abelian group, Matematychni Studii \textbf{38} (2012), 124--138.

\bibitem{Ga3}
 S.~Gabriyelyan, On a generalization of Abelian sequential groups, Fund. Math. \textbf{221} (2013), 95--127.

\bibitem{Ga4}
 S.~Gabriyelyan, Bounded subgroups  as a von Neumann radical of an Abelian group, Topology Appl. \textbf{178} (2014), 185--199.



\bibitem{Gra}
{ M.~Graev}, Free topological groups, Izv. Akad. Nauk SSSR Ser. Mat.  \textbf{12} (1948), 278--324.


\bibitem{HR1}
{ E.~Hewitt, K.~A.~Ross,}  Abstract Harmonic Analysis, Vol. I,
2nd ed. Springer-Verlag, Berlin, 1979.



\bibitem{Nie}
{ J.~Nienhuys}, Construction of group topologies on Abelian groups, Fund. Math.  \textbf{75} (1972), 101--116.

\bibitem{Nob}
{ N.~Noble}, $k$-groups and duality, Trans. Amer. Math. Soc. \textbf{151} (1970), 551--561.

\bibitem{ZP1}
{ I.~V.~Protasov, E.~G.~Zelenyuk}, Topologies on abelian groups,
Math. USSR Izv.  \textbf{37} (1991), 445--460. Russian original: Izv. Akad.
Nauk SSSR  \textbf{54} (1990), 1090--1107.

\bibitem{ZP2}
{ I.~V.~Protasov, E.~G.~Zelenyuk}, Topologies on groups determined
by sequences, Monograph Series, Math. Studies VNTL, L'viv, 1999.


\bibitem{Rac}
{ S.~U.~Raczkowski}, Totally bounded topological group topologies
on the integers, Topology Appl. \textbf{121} (2002), 63--74.






\smallskip


\end{thebibliography}

\end{document}